\documentclass[a4paper,11pt]{article}

\usepackage[scale=0.75,a4paper]{geometry}
\usepackage[parfill]{parskip}
\usepackage{graphicx}
\usepackage{amssymb}
\usepackage{epstopdf}
\usepackage{float}
\usepackage{mathtools}
\usepackage{listings}

\usepackage{libertine}
\usepackage{amsmath}

\AtBeginDocument{%
  \DeclareSymbolFont{mytt}{OT1}{cmtt}{m}{n}%
  \SetSymbolFont{mytt}{bold}{OT1}{cmtt}{bx}{n}%
  \DeclareSymbolFontAlphabet{\mathtt}{mytt}%
}

\usepackage[normalem]{ulem}
\usepackage{mathrsfs}
\usepackage{bbm}
\usepackage{todonotes}
\usepackage{units}
\usepackage{comment}

\usepackage{amsthm}

\usepackage{multirow}
\usepackage[title]{appendix}
\usepackage{textcomp}
\usepackage{manyfoot}
\usepackage{booktabs}

\usepackage{enumitem}
\usepackage{algorithm2e}
\usepackage{wasysym}

\usepackage{placeins}

\usepackage{appendix}

\usepackage{tcolorbox}

\usepackage[maxnames=6, backend=bibtex, url=false, isbn=false, doi=false, style=numeric]{biblatex}
\renewbibmacro{in:}{}
\usepackage[affil-sl]{authblk}

\newtheorem{theorem}{Theorem}[section]
\newtheorem{corollary}[theorem]{Corollary}
\newtheorem{proposition}[theorem]{Proposition}

\newtheorem{lemma}[theorem]{Lemma}

\newtheorem*{theorem*}{Theorem}

\theoremstyle{remark}
\newtheorem{remark}[theorem]{Remark}

\AtBeginEnvironment{remark}{%
  \pushQED{\qed}%
}
\AtEndEnvironment{remark}{\popQED}

\theoremstyle{definition}

\AtBeginEnvironment{definition}{%
  \pushQED{\qed}%
}
\AtEndEnvironment{definition}{\popQED}

\newcommand{\1}{\mathbbm{1}}

\newcommand{\de}{\operatorname{d}}

\DeclareMathOperator{\e}{\mathbb{E}}
\DeclareMathOperator{\p}{\mathbb{P}}
\DeclareMathOperator{\te}{\widetilde {\mathbb{E}}}
\DeclareMathOperator{\tp}{\widetilde {\mathbb{P}}}

\DeclareMathOperator{\T}{{\mathbb{T}_d}}

\usepackage{subcaption}

\usepackage{nameref}
\makeatletter
\let\orgdescriptionlabel\descriptionlabel
\renewcommand*{\descriptionlabel}[1]{%
  \let\orglabel\label
  \let\label\@gobble
  \phantomsection
  \edef\@currentlabel{#1\unskip}%
  \let\label\orglabel
  \orgdescriptionlabel{#1}%
}
\makeatother

\usepackage{hyperref}
\usepackage{orcidlink}

\title{Percolation of the contact process on the regular tree}
\author[1]{John Fernley \orcidlink{0000-0002-6635-4341}}
\affil[1]{CRiSM, Department of Statistics, University of Warwick,
United Kingdom. 
}
\author[2]{Emmanuel Jacob \orcidlink{0000-0001-5232-6668}}
\affil[2]{
Department of Mathematics,
New York University Shanghai, China.
}
\date{\today}
  
\begin{document}
\maketitle

\begin{abstract}
The contact process on the regular  tree $\T$ when $d\geq 3$ has the two phase transitions of global and of local survival, found by Pemantle and Liggett at values $\lambda_1$ and $\lambda_2$.
We start the system with every vertex infected and let it relax to what is known as the upper invariant infection. In this stationary state, $\lambda_p$ is the critical value beyond which the infected vertices can percolate
 through $\T$, and $\lambda_{p^\complement}$ is the parameter before which the healthy vertices can percolate. We find these are both distinct phase transitions
\[0<\lambda_1<\lambda_p<\lambda_2<\lambda_{p^\complement}<+\infty\]
on $\T$ when $d\geq 7$. The most interesting of these comparisons is $\lambda_1<\lambda_p$, which we find for all $d\geq 3$. This comparison $\lambda_1<\lambda_p$ is a long-standing open question on $\mathbb{Z}^d$ with $d\geq 2$ and was not yet found on any other graphs except where $\lambda_p$ is infinite.
\end{abstract}


\section{Introduction}

The random regular graph is a simple model for a sparse social network, and the contact process is a simple model for a spreading epidemic. It was found in \cite{MR3492935,MR3693520} that the contact process with infection rate $\lambda>0$ has two phases on this graph: for $\lambda<\lambda_1(\T)$ the epidemic period is logarithmic in the size of the graph whereas for $\lambda>\lambda_1(\T)$ it is exponential. This transition point is denoted $\lambda_1(\T)$ because it is the same parameter value $\lambda_1(\T)\in[\tfrac{1}{d-1},\tfrac{1}{d-2})$ identified by \cite{MR1188054} at which an infinite epidemic becomes possible on the regular tree $\T$, which is the local limit of the random regular graph. With this perspective of the tree as the local environment of the network, it is natural to expect that the metastable epidemic has the same features in both contexts. In this work, therefore, we look at structural features of a stationary infection on the regular tree $\T$ in order to make (unproved) inferences about the structural features of a metastable epidemic: in particular, about the connectivity properties of the infected set.

Consider the rooted graph $G=(V,E)$ to be either the $d-$dimensional lattice $\mathbb Z^d$, or the $d-$regular tree $\T$, with root $\circ$. The contact process on $G$ with infection rate $\lambda>0$ is a continuous-time Markov process $(\xi_t)_{t\ge 0}$ with values in $\{0,1\}^V$ describing the spread of an infection on $G$, with the interpretation that $\xi_t(x)=1$ if the individual $x$ in $V$ is infected at time $t$, while $\xi_t(x)=0$ if it is healthy.
For convenience, we often identify the configuration $\xi_t$ with the set of infected vertices $A_t=\xi_t^{-1}(\{1\})$.

The dynamics of the process are as follows: each infected vertex recovers at rate 1, and each edge between an infected vertex and a healthy one transmits the infection to the healthy vertex at rate $\lambda$. One way to make this description rigorous uses the \emph{graphical representation} of the contact process. Attach to each vertex a Poisson process of intensity 1 describing recoveries of the vertex if infected, and each directed edge a Poisson process of intensity $\lambda$ of \emph{infection arrows}, describing the transmission of the infection along the edge, assuming the head vertex is infected and the tail vertex healthy. 
Given an initial set of infected vertices $A$, the configuration
$\xi_t^A$ at time $t$ consists of all vertices reachable at time $t$ via infection paths starting from $A$ at time 0, 
following infection arrows and avoiding recoveries. Special cases include $\xi_t^{\{\circ\}}$ (infection starting only at the root) and $\xi_t^V$ (process starting from full infection).


Define the weak survival probability
\[
\rho(\lambda)=\p\big(\,
\forall t\geq 0, \; \xi_t^{\{\circ\}}\neq \varnothing
\big)
\]
and the strong survival probability
\[
\rho_2(\lambda)=\p\big(\, \xi_t^{\{\circ\}}(\circ)=1 \text{ for arbitrarily large }t
\big)\le \rho(\lambda).
\]
Because $\rho$ and $\rho_2$ are necessarily non-decreasing in $\lambda$ we can define the critical parameters
\begin{align*}
\lambda_1&=\sup\left\{
\lambda : \rho(\lambda)=0
\right\},\\
\lambda_2&=\sup\left\{
\lambda : \rho_2(\lambda)=0
\right\},
\end{align*}
and we automatically have $\lambda_1\le \lambda_2$.
When $G$ is either $\mathbb Z^d$ or $\T$, it is well-known that these phase transitions are nontrivial, namely $\lambda_1$, $\lambda_2$ are in $(0,+\infty)$. Moreover, the phase transition is continuous at $\lambda=\lambda_1$, and actually the weak survival probability $\rho(\lambda)$ is continuous in $\lambda$ on the whole line.


A striking difference between the lattice and the tree is that we have a single phase transition with $\lambda_1=\lambda_2$ in the case of lattice $\mathbb Z^d$, while we have a double phase transition with $\lambda_1<\lambda_2$ in the case of the $d-$regular tree $\T$ with $d\ge 3$ (note of course that $\mathbb{T}_2=\mathbb{Z}$).

The main interest in this paper lies however in a different phase transition, regarding percolation of the upper-invariant measure. The upper invariant measure is the probability measure $\nu_\lambda$ on sets $\eta \subset V(G)$ which can be defined as the limit distribution of the set of infected vertices on the graph started with full infection. By time-reversal duality of the contact process, we can also construct a random set $X_\lambda\sim\nu_\lambda$ by declaring a vertex $v$ to be in $X_\lambda$ if the infection started from only $v$ infected survives in the graphical representation, namely
\[
X_\lambda=\{v \in V \; : \; \forall t\ge 0, \; \xi_t^{\{v\}}\ne \emptyset\}.
\]
Say that $X_\lambda$ percolates if it contains an infinite component (note this is an increasing property), and define the ``percolation of the upper invariant measure'' threshold $\lambda_p$ as
\begin{align*}
\lambda_p&=\sup\left\{
\lambda : \p(X_\lambda \textrm{ percolates})=0
\right\}\\
&=\sup\left\{
\lambda : \p(\circ \textrm{ is in an infinite component of }X_\lambda)=0
\right\} .
\end{align*}
Note that for each $v\in V$, we have $\p(v\in X_\lambda)=\rho(\lambda)$, and actually $\nu_\lambda$ is nontrivial,  namely different from the Dirac mass on the all healthy configuration, if and only if $\rho(\lambda)>0$. We thus have the trivial inequality $\lambda_1\le \lambda_p$.


The percolation properties of $X_\lambda$ have been first considered in~\cite{MR2199800}, which proves in particular that for the $d$-dimensional lattice $\mathbb Z^d$, writing $\lambda_p(\mathbb Z^d)$ to stress the dependence on the graph of the percolation threshold, we have
\begin{align*}
\lambda_p(\mathbb Z)&= +\infty,\\
\lambda_p(\mathbb Z^d)&\le 6.25 <+\infty\qquad \text{ for }d\ge 2. 
\end{align*}
The same paper also raises the natural question whether we have $\lambda_1<\lambda_p$ on $\mathbb Z^d$ when $d\ge 2$, which remains a difficult open problem. To the best of our knowledge, such a result has previously never been proven for any graph with $\lambda_p<+\infty$.  In \cite{MR4452653} a variant question was answered about percolation through $\mathbb{Z}^d$ where the contact process in essence lives on a different graph and can use extra edges between all pairs of $\ell^\infty$ distance at most $R$, but the percolation considered is nonetheless only through $\mathbb{Z}^d$ edges.

In this paper, we study $\lambda_p$ in the case of the $d-$regular tree $\T$. The appearance of a distinct phase transition is an interesting mathematical observation in its own right, but moreover we see the upper invariant measure on the tree as analogous to the metastable infection on the random regular graph, a correspondence exploited in e.g. \cite{MR5091097}. Understanding the structure of a stationary infection on the regular tree is thus a very natural point of interest in a natural simplified model of an epidemic. In particular, more structural understanding of the metastable epidemic could lead to understanding the most likely path to recovery and so the length of the epidemic period \cite[Section 3.6.5]{aldous-fill-2014}.

\section{Results}

Our results on $\lambda_p$ are as follows:

\begin{theorem}\label{theorem_lambda_p}
	On the $d-$regular tree $\T$ with $d\ge 3$, the upper-invariant percolation threshold $\lambda_p$ satisfies:
	\begin{enumerate}[label=(\alph*), ref=\thetheorem(\alph*)]
		\item\label{theorem_lambda_p:a} the lower bound 
		\[\lambda_1<\lambda_p\]
		\item\label{theorem_lambda_p:b} the upper bound 
		\[\lambda_p \le \frac 1 {d-2}\]
		or for $d\geq 7$ the upper bound
		\[
		\lambda_p\le \frac 1 {d-2+\tfrac{1}{91}}.
		\]
	\end{enumerate}
\end{theorem}

%
%

\begin{remark}
	From 
	Theorem~\ref{theorem_lambda_p:b} and \cite[Theorem 2.2]{MR1188054}, we immediately deduce that $\lambda_p<\lambda_2$ on the $d$-regular tree when $d\ge 5$. We do not necessarily expect a general result on the comparison between $\lambda_p$ and $\lambda_2$, and in particular it is unclear to us how they compare when $d$ is $3$ or $4$. Note that for $d=2$ the inequality is in the reverse order by the identification of the $2$-regular tree with $\mathbb Z$ and previous discussion.	
\end{remark}

\begin{remark}
	By the bounds
	\[
	\frac 1{d-1}\le \lambda_1< \lambda_p\le \frac 1 {d-2}
	\]
	we see that that $\lambda_1$ and $\lambda_p$ are quite close and both asymptotically equivalent to $d^{-1}$ as $d\to +\infty$. It is then natural to question whether we can have more precise asymptotics for these thresholds. Actually, the more precise upper bound on $\lambda_1$ provided in \cite[Theorem 2.2]{MR1188054} gives
	\[1\le \liminf_{d\to\infty} d^2(\lambda_1-d^{-1})\le \limsup_{d\to\infty} d^2(\lambda_1-d^{-1})\le\frac 4 3,
	\]	
	but doesn't provide a second order asymptotic. Similarly, our upper bound on $\lambda_p$ for $d\ge 7$ gives
	\[1\le \liminf_{d\to\infty} d^2(\lambda_p-d^{-1})\le \limsup_{d\to\infty} d^2(\lambda_p-d^{-1})\le2- \frac 1 {91}.
	\]	
	Our proof could certainly be refined to replace the upper bound $2-1/91$ 
	with a smaller value,
	however we do not expect that it could be decreased up to $4/3$ and certainly not up to $1$.	
\end{remark}

\cite{MR4452653} also discusses the percolation threshold for the set of healthy vertices under the upper invariant measure, namely $\lambda_{p^\complement}$, defined by
\begin{align*}
\lambda_{p^\complement}(G)&=\inf\left\{
\lambda : \p\big(X_\lambda^\complement \textrm{ percolates}\big)=0
\right\}\\
&=\inf\left\{
\lambda : \p\big(\circ \textrm{ is in an infinite component of }X_\lambda^{\complement} \, \big)=0
\right\},
\end{align*}
where $X_\lambda^\complement=V \setminus X_\lambda$. 
Our results on $\lambda_p^\complement$ are as follows:

\begin{theorem}\label{theorem_lambda_p_complement}
	On the $d-$regular tree $\T$ with $d\ge 3$, the upper-invariant healthy vertices percolation threshold $\lambda_{p^\complement}$ satisfies:
	\begin{enumerate}[label=(\alph*), ref=\thetheorem(\alph*)]
		\item\label{theorem_lambda_p_complement:a} the lower bound 
		\[\lambda_1<\lambda_{p^\complement}
		\]
		\item\label{theorem_lambda_p_complement:b} the lower bound 
		\[
		1-\frac  2 d \le \lambda_{p^\complement}
		\]
		\item\label{theorem_lambda_p_complement:c} the upper bound 
		\[
		\lambda_{p^\complement}\le 1+\frac 1 {d-2}.
		\]
	\end{enumerate}
\end{theorem}

%
%
%
%
%

\begin{remark}
	For $d=3$ we have $\lambda_1\ge \frac 1 2$ so the better lower bound on $\lambda_{p^\complement}$ is provided by Theorem~\ref{theorem_lambda_p_complement:a}. Instead, for $d\ge 4$, we have $\lambda_1<\frac 1 2$ and the better lower bound is provided by Theorem~\ref{theorem_lambda_p_complement:b}.
\end{remark}

\begin{remark}
	By  Theorem~\ref{theorem_lambda_p_complement:c} 
	and \cite[Part I Theorem 4.65]{MR1717346}, we immediately deduce that $\lambda_2<\lambda_{p^\complement}$ on the $d$-regular tree when $d\ge 7$. Again, it is unclear whether the same result holds true when $3\le d \le 6$.
\end{remark}

\begin{remark}
	Regarding the threshold exponents $\lambda_1, \lambda_2, \lambda_p, \lambda_{p^\complement}$ on $\T$ for large $d$, we have already discussed the following asymptotics:
	\begin{align*}
		\lambda_1, \lambda_p &\sim d^{-1}\\
		\lambda_{p^\complement}&\to 1.
	\end{align*}
	For completeness, we recall here that $\lambda_2$ is of order $d^{-1/2}$, thanks to the bounds on $\lambda_2$ provided in \cite{MR1717346},
	\[
	\frac 1 {2 \sqrt {d-1}}\le \lambda_2 \le \frac 1 {\sqrt {d-1}-1},
	\]
	and moreover from \cite[Theorem 2.2]{MR1188054} we have $\liminf_{d\rightarrow +\infty}\lambda_2\sqrt{d}\geq 2-\sqrt{2}$, but the correct asymptotic constant is yet unknown.
	\end{remark}

In Figure \ref{fig_curves} we show a simulation for the numerical value of the probability to be in an infinite component, either in the sense of a connected infected set or a connected healthy set. This simulation is based on the assumption that the density of the giant connected component in a uniformly selected regular graph is equivalent to the probability to be in an infinite component in the local limit of that random graph, i.e. the regular tree $\T$. Implicitly we also rely on the infection in each local environment being close to its metastable distribution at time $t=100$ (which we expect to happen in $O(1)$ time but the numerics are not clear). Regardless the figure suggests that both transitions are continuous, and it's possible as well that the probability to be in an infinite healthy component has a continuous derivative at $\lambda=\lambda_{p^\complement}$ while the probability to be in an infinite infected component at $\lambda=\lambda_{p}$ does not. We leave both questions as interesting open problems.

\begin{figure}[H]\centering
\includegraphics[width=0.8\textwidth]{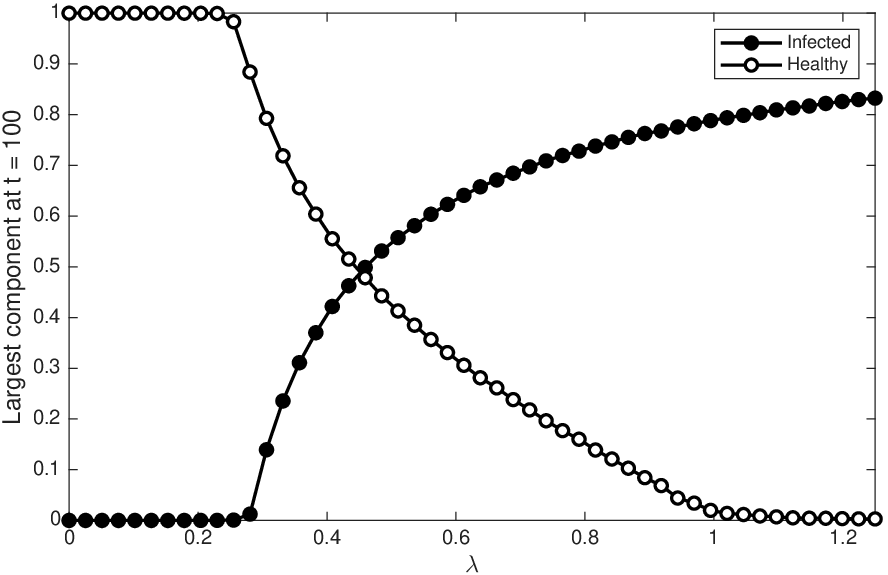}
\caption{$50$ trials on a uniform $5$-regular graph with $20,000$ vertices produce the above image, where we simulate the upper invariant infection by starting with every vertex infected and running to approximate metastability in $100$ units of time. Note that we can identify $\lambda_p$ in the bottom-left corner and $\lambda_{p^\complement}$ in the bottom-right. }\label{fig_curves}
\end{figure}

\section{Methods of proofs}

A natural technique to give bounds on $\lambda_p$ and $\lambda_{p^c}$ is to compare $\nu_\lambda$ with a product of Bernoulli  measures. Specifically, if we have the stochastic domination
\[
\nu_\lambda\succeq \bigotimes_{\T} {\rm Ber}\left(
p_\lambda
\right),
\]
with $p_\lambda> \frac 1 {d-1}$ (resp. $p_\lambda\ge \frac {d-2} {d-1}$), we immediately deduce that $X_\lambda$ percolates and thus $\lambda_p\le \lambda$ (resp. $X_\lambda^\complement$ does not percolate and thus $\lambda_{p^\complement}\le \lambda$).  \cite{MR2199800} shows the stochastic domination with $p_\lambda=1-4/\lambda$, showing finiteness of $\lambda_p$ and $\lambda_{p^\complement}$ (see Section 3, Example (c) where finiteness of $\lambda_p$ on $\T$ is stated explicitly), however this provides the bounds $\lambda_p\leq 4( \tfrac{d-1}{d-2} )$ and $\lambda_{p^\complement}\le 4(d-1)$, which are quite loose as compared to our results.

\cite{MR1457624} also proves that $X_\lambda$ percolates if the density of vertices in $X_\lambda$ is large enough, see \cite[Theorem 1.6]{MR1457624} for the case of bond percolation and last paragraph of \cite[Section 1]{MR1457624} for the case of site percolation. Beyond the upper-invariant measure of the contact process, the result applies to any automorphism invariant probability measure on $\{0,1\}^{\T}$.
Under the same assumption, however, an arbitrary low density of
 $1$'s is not sufficient to prevent percolation in general.

Let us now discuss briefly the main ingredients of the proofs of
Theorems~\ref{theorem_lambda_p} and~\ref{theorem_lambda_p_complement}:

\begin{itemize}
	\item 
	Theorem~\ref{theorem_lambda_p_complement:a} follows directly from~\cite{MR1457624} and the continuity of the phase transition at $\lambda_1$.
	\item 
	Theorem~\ref{theorem_lambda_p_complement:b}
	follows from the easy stochastic domination result
		\[
		\nu_\lambda\preceq \bigotimes_{\T} {\rm Ber}\left(
		\frac {\lambda d} {1+\lambda d}
		\right),
		\]
	obtained by the observation that a vertex in $X_\lambda$ must necessarily infect a neighbour before recovering. Thus healthy vertices percolate whenever 
	\[
	\frac {\lambda d}{1+\lambda d}<\frac {d-2}{d-1},
	\]
	yielding the result.
	\item For 
	Theorem~\ref{theorem_lambda_p_complement:c}, we consider $\nu_\lambda^{\mathbb Z}$ the marginal of $\nu_\lambda$ obtained by observing $X_\lambda$ only along a line in $\T$ (isomorphic to $\mathbb Z$, hence the notation), and prove
	\begin{equation} \label{line-SIRdomination}
	\nu_\lambda^{\mathbb Z}\succeq \bigotimes_{\mathbb Z} {\rm Ber}\left(
	1-\frac 1 {\lambda (d-2)}
	\right),
	\end{equation}
	by considering the infection process independently in each branch of $\T$  emanating from the vertices in the line.
	\item 
	The proof of Theorem~\ref{theorem_lambda_p:b} is more involved. We draw our inspiration from H\"aggstr\"om's approach in~\cite{MR1457624}, which relies on an application of the Mass Transport Principle on the  components of occupied vertices. In our case we can use the specific properties of the contact process to study the degree distribution of the vertices in those components. We obtain an exact formula for the first moment of this degree distribution, yielding the upper bound $1/(d-2)$ for $\lambda_p$, while a deeper analysis of the degree distribution allows to refine the upper bound for large $d$. Note that our argument does not require any control on the density of vertices in $X_\lambda$.
	
	Of course, using the density we could instead get upper bounds on $\lambda_p$ directly by the results of \cite{MR1457624}---however these would necessarily be weaker than the bounds of Theorem~\ref{theorem_lambda_p:b}. This is because we will see later that for $v\sim w$ and $\lambda>\lambda_1$ that
	\[
	\p(v,w\in X_\lambda)=\left(1-\frac{1}{\lambda d}\right)\p(v\in X_\lambda)\leq \left(1-\frac{1}{\lambda d}\right)^2.
	\]
	
	Therefore the H\"aggstr\"om condition 
	\[
	\left\{\lambda : \p(v,w\in X_\lambda)\geq \frac{2}{d}\right\}
	\subseteq
	\left\{\lambda : \left(1-\frac{1}{\lambda d}\right)^2\geq \frac{2}{d}\right\}
	=
	\left[\frac{1}{d-2}\left(1+ \sqrt{\frac{2}{d}}\right),\infty\right)
	\]
	gives an upper bound on $\lambda_p$ which is larger than $\tfrac{1}{d-2}$. H\"aggstr\"om's alternative vertex condition similarly gives a larger  upper bound than $\tfrac{1}{d-2}$
	\[
	\left\{\lambda : \p(v\in X_\lambda)\geq \frac{d}{2d-2}\right\}
	\subseteq
	\left[\frac{1}{d-2}\left(2- \frac{2}{d}\right),\infty\right).
	\]
	\item Finally, the proof of  
	Theorem~\ref{theorem_lambda_p:a} uses again the continuity of the phase transition at $\lambda_1$, and an explicit upper bound on the hitting probability of a vertex in the tree at $\lambda=\lambda_1$. A careful analysis of the infection trajectories allows to bound the probability that a segment of arbitrary length $k$ is entirely contained in $X_\lambda$ by $p_\lambda^k$, for some $p_\lambda$ which is strictly smaller than $1/(d-1)$ when $\lambda$ is sufficiently close to $\lambda_1$. We obtain this result without having any stochastic domination of $\nu_\lambda$ by a product of Bernoulli variables. 
\end{itemize}

At this point, the proofs of 
Theorems~\ref{theorem_lambda_p_complement:a} and \ref{theorem_lambda_p_complement:b} are already complete.
We prove 
Theorem~\ref{theorem_lambda_p_complement:c} in Section~\ref{sec:complement_upper_bound}, then 
Theorem~\ref{theorem_lambda_p:b} in Section~\ref{sec:upper_bound_MTP}, and finally 
Theorem~\ref{theorem_lambda_p:a} in Section~\ref{sec:lower_bound}.

\section{Upper bound on $\lambda_{p^\complement}$}
\label{sec:complement_upper_bound}

We start with the following lemma, to lower bound the survival probability in a particular subtree of $\T$. When $\lambda>\frac{1}{d-2}$, this is also a useful lower bound on the survival probability in the full tree.

\begin{lemma}
	For $d\ge 3$, consider the infinite tree $\T'$ where the root $\circ$ has degree $d-2$ and all other vertices have degree $d$ and thus $d-1$ children. Then the survival probability of the contact process starting with only the root infected is at least $1-1/ \lambda(d-2)$.
\end{lemma}

\begin{proof}
	Consider any configuration with a finite set $X\subset \T'$ of infected vertices. If $X$ contains $n$ vertices and $k$ internal edges, necessarily with $k\le n-1$, then the number of edges between $X$ and $X^\complement$ is
	\[
		\vert \partial X\vert = d n - 2k - 2 \1_{\circ \in X}\ge (d-2)n.
	\]
	Thus we can lower bound the number of infected vertices by a birth-death chain with
	\[
	\begin{split}
		Q_{n,n+1}&=\lambda n(d-2)\\
		Q_{n,n-1}&=n,
	\end{split}
	\]
	from which we immediately deduce that the extinction probability is at most $1/\lambda(d-2)$.
\end{proof}

Consider now an infinite line $\mathbb Z\subset \T$. Each vertex in this line belongs to a disjoint copy of $\T'$. You thus get~\eqref{line-SIRdomination} by finding survival independently for each vertex in its copy of $\T'$.

Further, a first moment method yields that the expected number of vertices in $\T$ at distance $n$ from the root and connected to the root in $X_\lambda^\complement$ is bounded by
\[
d(d-1)^{n-1}
\left(
\frac{1}{\lambda(d-2)}
\right)^n.
\]
This term tends to 0 as $n$ tends to infinity when
\[
\frac{1}{\lambda(d-2)}
<
\frac{1}{d-1}
\iff
\lambda > 1+\frac{1}{d-2}.
\]
Under this condition, the probability that the root is in an infinite component of healthy vertices is 0, and thus $\lambda\ge \lambda_{p^\complement}$. The result Theorem~\ref{theorem_lambda_p_complement:c} follows.

\section{Upper bounds on $\lambda_p$}
\label{sec:upper_bound_MTP}

\subsection{Bound for $d \geq 3$ by the Mass Transport Principle}

First in this section, we detail a mass transport principle to obtain an exact expression for the expected degree within both finite and infinite components. After we will show that the mean is actually explicit and thus obtain a lower bound on $\lambda_p$.

Observe that the law of $X_\lambda$ is invariant under the isometries of $\T$, which allows to express and use a version of the Mass Transport Principle. Before doing so, we introduce some notation for the graph. Identify $X_\lambda$ with the subgraph of $\T$ induced by keeping only the vertices in $X_\lambda$. For a vertex $v$ in $X_\lambda$, write 
\[D_v= \sum_{w \sim v} \1_{w\in X_\lambda}\]
for its degree in $X_\lambda$. Write $C_v=C_v(X_\lambda)$ for the connected component of $v$ in this graph, which is thus a (finite or infinite) tree. For an edge $w\sim v \in X_\lambda$, write $C_{v\backslash w}=C_v(X_\lambda \setminus\{w\})$ for the component of $v$ obtained by removing $w$, so that the full component $C_v$ is recovered by taking the union of the two disjoint trees $C_{v\backslash w}$ and $C_{w\backslash v}$ and adding the edge $\{v,w\}$.

The edge $\{v,w\}$ is called the \emph{centroid edge} of $C_v$ if $C_v$ is a finite tree and
\[
|C_{v\backslash w}|=|C_{w\backslash v}|,
\]
namely the two trees contain the same number of vertices. A vertex $v$ is the \emph{centroid vertex} if $C_v$ is a finite tree and the trees $C_{v\backslash w}$ all contain strictly fewer than half the vertices of $C_v$. A finite tree contains either a centroid vertex or a centroid edge, while an infinite tree has neither. Finally, we write
\[D^\infty_v:=\#\{w\sim v \text{ in } X_\lambda, |C_{w\backslash v}|=\infty\}\]
for the number of infinite branches at $v$.

\begin{proposition}\label{prop:ApplyMTP}
	Write $\tp{}$ for the probability measure $\p$ conditionally on $\circ\in X_\lambda$. We have:
	\begin{multline*}
	\te(D_\circ)=2-2 \tp(\circ \text{ is the centroid vertex of }C_\circ)\\-\tp(\circ \text{ belongs to the centroid edge of }C_\circ) +\te\big((D^\infty_\circ-2)_+\big).
	\end{multline*}
\end{proposition}

In particular, if $\te(D_\circ)>2$ then we must have positive probability of an infinite component at $\circ$.

\begin{proof}
	The result relies only on the Mass Transport Principle (MTP) satisfied by $X_\lambda$. We use the transport function $f(v,w)$ defined as follows:
	\[
	f(v,w)=
	\begin{cases*}
		1& if $v \sim w \in X_\lambda$ and $|C_{v\backslash w}|>|C_{w\backslash v}|,$\\
		D_v- \#\{u\sim v, |C_{v\backslash u}|>|C_{u\backslash v}|\}  & 
		 if  $w=v\in X_\lambda,$\\
		0 &else,
	\end{cases*}
	\]
	where in the first case it should be understood  that $f(v,w)=1$ also in the case that only $|C_{v\backslash w}|$ is infinite, but not in the case that both are infinite. 
	By the definition of $f(\circ,\circ)$, the total mass sent by $\circ$ is
	$\sum_v f(\circ,v)= D_\circ\1_{\circ \in  X_\lambda}$, while it is not hard to see that the total mass received by $\circ$ is:
	\[
	\sum_v f(v,\circ)=\begin{cases*}
		0 &if $\circ$ is in  $X_\lambda$ and is the centroid of $C_\circ$,\\
		1 & if $\circ$ is in $X_\lambda$ and belongs to the centroid edge of $C_\circ$,\\
		D^\infty_\circ \vee 2&in all other cases with $\circ\in X_\lambda$,\\
		0 & else.
	\end{cases*}
	\]
	By the Mass Transport Principle, $\e\big(\sum_v f(\circ,v)\big)=\e\big(\sum_v f(v,\circ)\big)$. Writing down this equality and dividing by the probability of $\circ \in X_\lambda$, we immediately get Proposition~\ref{prop:ApplyMTP}.
\end{proof}

\begin{remark}
	In finite components the mass is sent from the centroid (vertex or edge) towards the leaves, and every vertex but the centroid receives mass 2. 
	We could replace ``centroid'' by ``center'' by considering the depths of the trees $C_{v\backslash w}$ and $C_{w\backslash v}$ rather than  their number of vertices. 
	We could also dispatch this defect over the whole component, so we have three equivalent expressions:
	\begin{align*}
		2\te\left(\frac 1 {|C_\circ |}\right)&=2 \tp(\circ \text{ is the centroid of }C_\circ)+\tp(\circ \text{ belongs to the centroid edge of }C_\circ)\\
		&= 2 \tp(\circ \text{ is the center of }C_\circ)+\tp(\circ \text{ belongs to the center edge of }C_\circ).
	\end{align*}
	However, in view of our application of Proposition~\ref{prop:ApplyMTP}, none of the three expressions seems to be more useful than the others.
\end{remark}

Surprisingly, we can find an exact expression for the mean degree $\te(D_\circ)$.

\begin{proposition}\label{prop_mean_degree}
If $\lambda>\lambda_1$ and $d\geq 2$ we find
\[
\te(D_\circ)=d-\frac{1}{\lambda}.
\]
\end{proposition}

\begin{proof}
We write in this section $\rho(\{\circ\})=\p({\circ\in X_{\lambda}})$ for the survival probability of the infection from $\circ$, and $\rho(\{\circ,v\})$ for the survival probability from $\{\circ,v\}$.  Observe by considering the first infection or recovery that if $v\sim \circ$ then
\[
\rho(\{\circ\})
=
\sum_{w\sim \circ}
\frac{\lambda}{1+\lambda d}\,
\rho(\{\circ,w\})
=
\frac{\lambda d}{1+\lambda d}\,
\rho(\{\circ,v\}).
\]

Moreover by inclusion--exclusion
\[
\p({\circ,v\in X_{\lambda}})
=2\rho(\{\circ\})-\rho(\{\circ,v\})
=\left(
1-\frac{1}{\lambda d}
\right)\rho(\{\circ\}).
\]

It follows that
\[
\te(D_\circ)
=\sum_{v\sim \circ}\frac{\p({\circ,v\in X_{\lambda}})}{\p({\circ\in X_{\lambda}})}
=d\left(
1-\frac{1}{\lambda d}
\right),
\]
as claimed.
\end{proof}

We next use this mean degree, with the mass transport principle, to argue that there is an infinite component.

\begin{corollary}\label{cor:perc}
On $\T$ with $d\geq 3$ and
$\displaystyle
\lambda>
\frac{1}{d-2+2\tp(D_\circ=0) 
}
$ we must have infinite components with positive probability.
\end{corollary}

\begin{proof}
We just apply Proposition \ref{prop:ApplyMTP}.
Note that if $D_\circ=0$ then $\circ$ is necessarily a centroid, and so
\[
\te(D_\circ)-2+2 \tp(D_\circ=0) \leq \te\big((D^\infty_\circ-2)_+\big).
\]

From Proposition \ref{prop_mean_degree} above we know that in fact $\te(D_\circ)=d-\frac{1}{\lambda}$. So, by rearranging, the right-hand side $\te\big((D^\infty_\circ-2)_+\big)$ must be positive with the claimed condition, showing the existence of infinite components with positive probability.
\end{proof}

\begin{remark}
	Actually more can be said regarding the number of topological ends of the infinite components, based on the simple observation that $(D^\infty_\circ-2)_+$ can be nonzero only if $\circ$ belongs to an infinit component with at least three topological ends. By \cite{MR1457624}, the number of topological ends must be 1, 2 or infinity, thus we find a component with infinitely many topological ends. By \cite[Theorem 1.2]{MR1457624}, we are also finding infinitely many components with infinitely many ends.


	
	It seems impossible for $1$- or $2$-ended components to occur in the upper invariant measure and \cite[Theorem 1.3]{MR1457624} gives a condition for this result, but this condition seems hard to verify. Further,  if it is of interest, \cite[Lemma 2.3]{MR1457624} gives a path to an explicit lower bound on the growth rate of all these infinite-ended components.
\end{remark}

\subsection{Improved bound for $d\geq 7$}

We consider $d\ge 7$ and restrict $\lambda$ to the interval $[\tfrac{1}{d-1},\tfrac{1}{d-2}]$ where we know we can find $\lambda_p$. We aim to show the lower bound $\tp(D_\circ=0)\geq {1}/{181}$, which with Corollary~\ref{cor:perc} will prove the refined upper bound since $2/181>1/91$.



\begin{lemma}\label{lem_root_occupation}
When $\lambda\leq\frac{1}{d-2}$ and $d\geq 7$ we have
\[
\e_\circ(\text{infectious period of }\circ)
\leq
1+
\frac{2 d}{(d-1) \left(d-2 \sqrt{d-1}+2\right)}.
\]
\end{lemma}

\begin{proof}
We upper bound the contact process by the branching random walk, in which every particle dies at rate $1$, and gives birth to a child particle at each neighbouring site at rate $\lambda$ (and thus gives birth at total rate $\lambda d$). This branching random walk does not survive locally when $\lambda\le \frac{1}{2\sqrt{d-1}}$, which is the case for all $\lambda\le\frac 1 {d-2}$ when $d\ge 7$. 
We apply the supermartingale of \cite[Part I Equation 4.6]{MR1717346} with weight $\theta=\frac{1}{\sqrt{d-1}}$ to find
\[
\e_v\left(
\int_0^{+\infty}
\sum_{x \in \T} \xi_t(x) \theta^{\ell_v(x)}
{\rm d}t
\right)\leq \frac{1}{1-2\lambda\sqrt{d-1}}
\]
where $\ell_v(v)=0$ and every vertex has $d-1$ neighbours with $\ell_v$ one larger and $1$ neighbour with it one smaller. Now our root $\circ$ expects to infect
\[
\frac{\lambda d}{1+\lambda}
\]
such neighbours $v$, and each of them gives $\circ$ weight $\theta^{-1}$ in its supermartingale. This produces
\[
\e_v(\text{infectious period of }\circ)
\leq
\theta\left(
\frac{1}{1-2\lambda\sqrt{d-1}}
-1\right)
=
\frac{2\lambda}{1-2\lambda\sqrt{d-1}}
\]
and so by summation
\[
\begin{split}
\e_\circ(\text{infectious period of }\circ)
&\leq
1+
\frac{\lambda d}{1+\lambda}
\cdot
\frac{2\lambda}{1-2\lambda\sqrt{d-1}}\\
&\leq
1+
\frac{2 d}{(d-1) \left(d-2 \sqrt{d-1}+2\right)}
\end{split}
\]
in the final line inserting again the worst case parameter $\lambda = \frac{1}{d-2}$. 
\end{proof}

%

In the following lemma we compare survival probabilities on the $(d-1)$-ary and  $d$-regular trees. The natural way to lower bound survival on the $(d-1)$-ary tree is by the SIR lower bound, but note that when $\lambda\leq \frac{1}{d-2}$ that SIR process has survival probability $0$.

\begin{lemma}\label{lem_dary}
On the $(d-1)$-ary tree with initially one infected vertex at the root (i.e. the tree in which every vertex has degree $d$ except the root which has degree $d-1$) when $\lambda\leq \frac{1}{d-2}$ and $d\geq 7$ we find survival probability at least
\[
\left(
1-\frac{1}{d}+O\left(
\frac{1}{d^2}
\right)
\right)\rho(\{v\}).
\]
\end{lemma}

\begin{proof}
Consider the CP-BRW of \cite{MR4771965} in $\T$ with two sinks at two adjacent vertices $v$ and $w$. Immediately we can write
\[
\begin{split}
\rho(\{v\}) &\leq \p_v(\text{path to $\infty$ exists in the CP-BRW})\\
&\leq \e_v(\text{CP-BRW particles which have some path to $\infty$})=:\mu.
\end{split}
\]

We can then solve for this mean, writing $\sigma$ for the survival probability on the $(d-1)$-ary tree and recalling Lemma \ref{lem_root_occupation}
\[
\mu\leq \sigma +
\lambda\left(
1+
\frac{2 d}{(d-1) \left(d-2 \sqrt{d-1}+2\right)}
\right)
\mu.
\]

Rearranging and setting $\lambda= \frac{1}{d-2}$, we find
\begin{equation}\label{eq_prod1}
\sigma \geq \left(
1-\frac{1}{d-2}\left(
1+
\frac{2 d}{(d-1) \left(d-2 \sqrt{d-1}+2\right)}
\right)
\right)
\rho(\{v\})
\end{equation}
which is $(1-d^{-1}+O(d^{-2}))\rho(\{v\})$ as claimed.
\end{proof}

In \cite{MR1717346} $\beta$ is the geometric base of the hitting probability for both the branching random walk and the contact process, in different sections. Here we write $\tilde{\beta}$ to emphasise that this is the branching random walk version for which an explicit expression is available: we will need the contact process version in Section \ref{sec:lower_bound} for which we reserve $\beta$.

\begin{proposition}[{\cite[Part I Equation 4.17]{MR1717346}}]\label{prop_liggett_hitting}
The probability that the branching random walk from $\circ$ on $\T$ ever occupies a vertex $v$ at distance $\de(\circ,v)=k\geq 1$ is bounded by
\[
\tilde{\beta}(\lambda)^k
\quad\text{where}\quad\tilde{\beta}(\lambda):=\frac{1-\sqrt{1-4(d-1)\lambda^2}}{2(d-1)\lambda},
\] 
when $\lambda<\frac{1}{2\sqrt{d-1}}$.
\end{proposition}

\begin{proposition}\label{prop_isolated}
When $\lambda\leq\frac{1}{d-2}$ and $d\geq 7$ we find
$\tp(D_\circ=0)\geq {1}/{181}$.
\end{proposition}

\begin{proof}
We design a scenario which ensures that $\circ$ is an isolated vertex in $X_\lambda$. For a neighbour $v\sim \circ$, we introduce the following events:
\begin{itemize}
	\item $A_1(v)$: $v$ recovers before infecting any neighbour, and before the infection started in $\circ$ recovers or infects any neighbour.
	\item $A_2(v)$: $A_1(v)$ occurs, and then the infection started in $\circ$ infects $v$, before recovering or infecting any other neighbour.
	\item $A_3(v)$: $A_2(v)$ occurs, and then the infection that has reached $v$, survives in the $d$-ary tree with root $v$ obtained by removing $\circ$.
	\item $A_4(v)$: the infections started from all other neighbours $w\ne v$ of $\circ$, all die out and also don't infect the root.
\end{itemize}
	It should be clear that on the event $A_3(v)\cap A_4(v)$ we have that $\circ$ is an isolated vertex in $X_\lambda$ and thus $D_\circ=0$. Moreover, $A_3(v)$ is independent from $A_4(v)$, and the events $A_3(w)\cap A_4(w)$ are disjoint when we take all neighbours $w\sim \circ$. Therefore,
	\[
	\p(D_\circ=0)\ge \sum_{w\sim \circ} \p(A_3(w)\cap A_4(w))\ge d \p(A_3(v)) \p(A_4(v)).
	\]
	Moreover, we have
	\[
	\p(A_2(v))=\p(A_1(v))\p(A_2(v)|A_1(v))=\frac{1}{2+2\lambda d}\cdot \frac{\lambda}{1+\lambda d}
	\]
	and if we write $\p(A_3(v)| A_2(v))= \sigma $ as in \ref{lem_dary}, we have
	\[\tp(D_\circ=0)\ge d\cdot \frac{1}{2+2\lambda d}\cdot \frac{\lambda}{1+\lambda d}\cdot \frac{\sigma}{\rho(\{\circ\})}\cdot\p(A_4(v)), \] 
	and it remains to control $\p(A_4(v))$.
	
 For $w\sim\circ$, $w\neq v$, let $B_w$ be the event that the contact process started from $\{w\}$ becomes extinct without ever reaching $\circ$. Thus
\[
A_4(v)=\bigcap_{\substack{w\sim\circ\\w\neq v}}B_w.
\]
Before reaching $\circ$, the process started from $w$ uses only the graphical marks in the component of $\T\setminus\{\circ\}$ containing $w$, together with the arrows from $w$ to $\circ$. These collections of marks are disjoint for different $w$, and disjoint from the marks determining $A_3(v)$, so all of those events are mutually independent. 

Fix one such $w$ and couple its contact process with the dominating branching random walk. If this branching random walk becomes extinct without visiting $\circ$, then $B_w$ occurs. The branching random walk is supercritical when $\lambda d>1$, but it is not difficult to see that the branching random walk conditioned on extinction is just another branching random walk, in which each particle:
\begin{itemize}
	\item dies at rate $\lambda d$
	\item gives birth to a child particle at each neighbouring site at rate $1/d$ (and thus gives birth at total rate $1$).
\end{itemize}
Note that the total rate at which a particle dies or gives birth is $1+\lambda d$ in both processes, however the rates $1$ and $\lambda d$ of the unconditioned process are reversed to $\lambda d$ and $1$ for the process conditioned on extinction. 
%
%
After rescaling time so that the death rate is one, we recover the original unconditioned branching random walk, but with infection rate
$\tfrac{1}{\lambda d^2}$ instead of $\lambda$.
 Moreover,
\[
\frac{1}{\lambda d^2}<\frac1d<\frac{1}{2\sqrt{d-1}},
\]
so Proposition~\ref{prop_liggett_hitting}, applied under the extinction-conditioned law to the adjacent vertices $w$ and $\circ$, yields
\[
\p(B_w)
\geq \frac{1}{\lambda d}
\left(1-\tilde{\beta}\left(\frac{1}{\lambda d^2}\right)\right).
\]
Taking the product over the $d-1$ independent branches gives
\begin{equation}\label{eq_prod3}
\p(A_4(v))
\geq
\left(
\frac{1-\tilde{\beta}(\tfrac{1}{\lambda d^2})}{\lambda d}
\right)^{d-1}.
\end{equation}

Finally, \eqref{eq_prod1} gives $\sigma/\rho(\{\circ\})\geq A_d$, in the notation of Appendix~\ref{app:isolation_product}. We combine this with \eqref{eq_prod3} and apply Lemma~\ref{lem:isolation_product_bound} to obtain
\[
\tp(D_\circ=0)
\geq A_dG_d(\lambda)H_d(\lambda)
>\frac{1}{181}.
\]
\end{proof}


\section{Lower bound on $\lambda_p$}
\label{sec:lower_bound}

In this section, we consider some fixed $d\ge 3$, and we prove $\lambda_p>\lambda_1$. We first introduce some notation. For two vertices $u$ and $v$ in the regular tree, we write $u\rightsquigarrow v$ for the event that there exists and infection path from $u$ to $v$ in the graphical construction of the contact process (started at time 0), or in other words that the infection started with only $u$ reaches $v$ (namely will infect $v$ at some positive time). If $A$ is a set of vertices, we also write $u\rightsquigarrow A$ for the event that there is an infection path from $u$ to some $v$ in $A$.

We also write $(e_n)_{n\ge 0}$ for a distinguished path in the tree, so $e_n$ and $e_0$ are at graph distance $n$ from each other. Writing
\[
\beta(\lambda)=\lim_{n\to +\infty} \left(\p(e_0\rightsquigarrow e_n)\right)^{\frac 1 n},
\]
we will use repeatedly that at $\lambda=\lambda_1$ we have
\begin{equation}\label{ineq:hitting_beta_bound}
	\p\left(e_0\rightsquigarrow e_k)\right) \le \beta(\lambda_1)^k = \frac 1 {(d-1)^k}
\end{equation}
by~\cite{MR1717346}, (4.49) and Corollary 4.87.

Finally, we introduce 
\begin{equation}
	E_k= \bigcap_{j=1}^k \{e_j\rightsquigarrow \{e_0, e_{k+1}\}\}
\end{equation}
the event that there exist infection paths from each vertex in $\{e_1,\ldots,e_k\}$ to either $e_0$ or $e_{k+1}$.

\begin{proposition} \label{prop_path_escape}
	There exists $k\ge 1$ such that at $\lambda=\lambda_1$, we have 
	\[
	\p(E_k)<\frac 1 {(d-1)^k}.
	\]
\end{proposition}

\begin{remark}
	It would be possible to adapt the proof to show that $\p(E_k)$ is even asymptotically smaller than $1/(d-1)^k P(k)$ for any polynomial $P$, but we will not need such a strong result. Alternatively, we could fix $k$ and look at the asymptotic value of $\p(E_k)$ when $d$ is large, in which case we would obtain
	$c_k (d-1)^{-k}$ with
	\[
	c_k=\sum_{j=0}^k \frac 1 {j!} \frac 1 {(k-j)!}.
	\]
	We have $c_4=2/3<1$ but $c_3>1$, so for large $d$ the inequality of Proposition~\ref{prop_path_escape} holds for $k=4$ but not for $k=1, 2, 3$.
\end{remark}

\begin{proof}
	Note that by~\eqref{ineq:hitting_beta_bound} we have $\p(E_1)\le \p(e_1\rightsquigarrow e_0)+\p(e_1\rightsquigarrow e_2)\le \frac 2{d-1}$. This bound is too large only by a two multiplicative factor. For larger $k$, we first provide an easy similar bound. We claim that we have
	\[
	E_k\subset \bigcup_{j=0}^k\ \{e_j \rightsquigarrow e_0\} \circ \{e_{j+1}\rightsquigarrow e_{k+1}\},
	\]
	with the convention that for $j=0$ (resp $j=k$) the event $\{e_0\rightsquigarrow e_0\}$ (resp. $\{e_{k+1}\rightsquigarrow e_{k+1}\}$) is the full event and the event in the union is thus simply $\{e_{1}\rightsquigarrow e_{k+1}\}$ (resp. $\{e_k\rightsquigarrow e_0\}$). Here $\circ$ means disjoint occurence of the two events, which in this context means that the infection paths realizing the two events can be taken to be disjoint. The union bound and BK-inequality (see~\cite{BK1985} for the historical reference or~\cite[Theorem B.21]{MR1717346} for a discussion on how to apply it in the context of the contact process), together with~\eqref{ineq:hitting_beta_bound}, provide the following bound:
	\begin{align*}
	\p(E_k)&\le \sum_{j=0}^k \p(e_j \rightsquigarrow e_0) \p(e_{j+1}\rightsquigarrow e_{k+1})\\
	&\le  \sum_{j=0}^k \left(\frac 1 {d-1}\right)^{j} \left(\frac 1 {d-1}\right)^{k-j}\\
	&\le \frac {k+1}{(d-1)^k}.
	\end{align*}
	This time the upper bound is too large by a factor $k+1$. However, we didn't use at this point that for $E_k$ to be realized, you should also have infection paths from vertices different from $e_j$ and $e_{j+1}$ to $e_0$ or $e_{k+1}$.
	
	In the remainder of this proof, we condition on the event 
	\[ E_{k,j}:=\{e_j \rightsquigarrow e_0\} \circ \{e_{j+1}\rightsquigarrow e_{k+1}\}\] for some $j\in \{0,\ldots,k\}.$ 
	Let $A_j=\{e_j,e_{j+1}\}$ for $1\le j\le k-1$, $A_0=\{e_1\}$ and $A_k=\{e_k\}$. We write $T_j<T_{j,2}$ for the first two recovery marks or outgoing infection arrows encountered in the graphical exploration from $A_j$. $E_{k,j}$ is independent from this holding time and so, conditionally on $E_{k,j}$, the two successive holding rates for $T_j$ and $T_{j,2}-T_j$ are at most $2+2\lambda d$ and $3+3\lambda d$, respectively.
	
	For $k\ge4$, choose $k-4$ vertices $e_i$ with $i\in\{1,\ldots,k\}\setminus\{j-1,j,j+1,j+2\}$, and look at the probability that one of them recovers before infecting any neighbour and before time $T_{j,2}$. Writing $R_i$ for the first recovery time of vertex $e_i$ and $I_i$ for the first infection time of a neighbour (in the graphical construction), these are independent random variables,
	with $R_i\sim \textrm{Exp}(1)$ and 
	$I_i\sim \textrm{Exp}(\lambda d)$. 
	Vertex $e_i$ recovers before time $T_{j,2}$ if
	\[
	R_i<\min(I_i, T_{j,2}).
	\]
	If we condition on $E_{k,j}$ and on $T_{j,2}$, this has explicit conditional probability
	\[
	\frac 1 {1+\lambda d} \left(1- e^{-(1+\lambda d) T_{j,2}}\right),
	\]
	from which we could compute the probability that this doesn't happen for any vertex $e_i$, which is a necessary condition for the event $E_k$ to be realized. However, we can ease this computation by considering, for each vertex $e_i$, the standard Poisson point processes (from the graphical representation) of recoveries and infections up to time $T_{j,2}$, and then independent ones after this time, so as to have independence with the event $E_{k,j}$. We still write $R_i$ and $I_i$ for the modified ``first infection and recovery'', which are now independent from $E_{k,j}$ but for which we still have recovery of the vertex under the same condition
	\[
	R_i<\min(I_i, T_{j,2}).
	\]
	The advantage of doing this is that we can now first reveal for which of these $k-4$ vertices we have $R_i<I_i$, and write $N_k$ for the random number of vertices for which this is the case.
	For each such vertex, we still have that $\min(R_i,I_i)\sim \textrm{Exp}(1+\lambda d)$. Thus, the probability that $T_j$ and $T_{j,2}$ are smaller than all these times, conditionally on $N_k$, is at most
	\[
	\frac 2 {N_k+2} \cdot\frac 3 {N_k+3}\le \frac {150}{k^2}+\mathbbm 1_{N_k\le k/5}.
	\]
	Finally, observe that  $N_k$ follows a binomial distribution with parameters $k-4$ and $1/(1+\lambda d)$. Using now $\lambda=\lambda_1$ and $\lambda_1 d \le d/(d-2)\le 3$, we can ensure that $N_k$ is larger than $k/5$ with failure probability that is exponentially small in $k$. We thus have
	\[\p(E_k | E_{k,j})\le \e\left[\frac 2 {N_k+2} \cdot\frac 3 {N_k+3}\right]\le O(k^{-2}),
	\]
	and thus
	\[
	\p(E_k)\le O\left(\frac{k^{-1}}{(d-1)^k}\right),
	\]
	proving the result by taking $k$ large enough.
	\end{proof}

\begin{corollary}
	For any $d\ge 3$, we have $\lambda_p>\lambda_1$.
\end{corollary}

\begin{proof}
	We fix some $d\ge 3$, and use Proposition~\ref{prop_path_escape} to fix some $k$ with $\p(E_k)<(d-1)^{-k}$ when $\lambda=\lambda_1$.
	
	Observe that for any $\lambda$, we can conclude that the upper invariant measure contains only finite components if the probability that it contains a path of length $kn$ decays faster than $\frac{1}{d(d-1)^{k n-1}}$.
	
	
	For this, divide the path into $n$ consecutive $k$-uplets. Call one of these sets $\{e_1,\ldots,e_k\}$ and note that
	\[
	\{e_1,\ldots,e_k\in \nu_\lambda\}\subset
	\{e_1,\ldots,e_k \text{ have a path to }\{e_0,e_{k+1}\}\text{ or to $\infty$ avoiding }\{e_0,e_{k+1}\}\}.
	\]
	
	The events on the right-hand side are independent. So we can calculate the probability of one such event, and raise it to the power $n$ to bound the path probability.
	
	When $\lambda=\lambda_1$, the infection doesn't survive, so the probability of the event on the right-hand side reduces to the probability that $e_1,\ldots,e_k \text{ have a path to }\{e_0,e_{k+1}\}$, which by Proposition~\ref{prop_path_escape} is smaller than $(d-1)^{-k}$. However, the probability in the right-hand side is continuous in $\lambda$, and so we can choose some $\lambda>\lambda_1$ and still have
	\[
	p_\lambda(k):=\p(e_1,\ldots,e_k \text{ have a path to }\{e_0,e_{k+1}\}\text{ or to $\infty$ avoiding }\{e_0,e_{k+1}\})<(d-1)^{-k}.
	\]
	We deduce
	\[
		\e\left(
		\#\{\text{vertices connected through $\nu_\lambda$ at distance }kn\}
		\right)
		\leq
		d(d-1)^{kn-1}
		p_\lambda(k)^n
		\rightarrow 0
	\]
	and by taking $n$ large we see that the probability to be in an infinite component must be $0$.
\end{proof}

\begin{figure}[H]\centering
\includegraphics[width=0.6\textwidth]{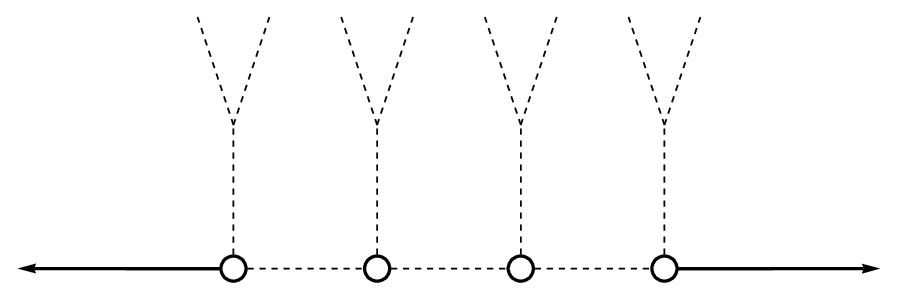}
\caption{The event $E_k$ of interest is that none of the infections from the $k$ circled vertices infect across the bold edges, sketched above for the $3$-regular tree with $k=4$.}
\end{figure}

\section*{Acknowledgements}

JF is grateful for the hospitality of NYU Shanghai where the majority of this work was done.

\printbibliography

\appendix

\section{Calculations for Proposition~\ref{prop_isolated}}
\label{app:isolation_product}

The given constant $1/181$ depends on some elementary numerical work. Put
\[
A_d:=1-\frac{1}{d-2}\left(
1+\frac{2d}{(d-1)(d+2-2\sqrt{d-1})}
\right),
\]
and let
\[
G_d(\lambda):=
\frac{d\lambda}{2(1+\lambda d)^2},
\qquad
H_d(\lambda):=
\left(
\frac{1-\tilde{\beta}(1/(\lambda d^2))}{\lambda d}
\right)^{d-1}.
\]

\begin{lemma}\label{lem:isolation_product_bound}
For every integer $d\geq7$ and $\lambda\in[1/(d-1),1/(d-2)]$,
\[
A_dG_d(\lambda)H_d(\lambda)>\frac1{181}.
\]
\end{lemma}

\begin{proof}
First we justify taking the endpoint in $\lambda$. Write $x:=\lambda d>1$. Then
\[
G_d(\lambda)=g(x):=\frac{x}{2(1+x)^2},
\qquad
g'(x)=\frac{1-x}{2(1+x)^3}<0.
\]

For the base of $H_d$, put
\[
s(x):=\sqrt{1-\frac{4(d-1)}{d^2x^2}}.
\]

Using the definition of $\tilde{\beta}$ and differentiating gives
\[
\frac{1-\tilde{\beta}(1/(dx))}{x}
=\frac1x-\frac{d(1-s(x))}{2(d-1)},
\qquad
\frac{\mathrm d}{\mathrm dx}
\frac{1-\tilde{\beta}(1/(dx))}{x}
=\frac1{x^2}\left(\frac{2}{dxs(x)}-1\right)<0,
\]
because $dxs(x)=\sqrt{d^2x^2-4(d-1)}\geq d-2>2$. Hence both $G_d$ and $H_d$ are minimized at $\lambda_*:=1/(d-2)$. At this endpoint write
\[
B_d:=G_d(\lambda_*)=\frac{d(d-2)}{8(d-1)^2}
\]
and
\[
R_d:=H_d(\lambda_*)^{1/(d-1)}
=\frac{d^2+\sqrt{d^4-4(d-1)(d-2)^2}-6d+4}{2d(d-1)}.
\]
It remains to bound $F_d:=A_dB_dR_d^{d-1}$.

We first record two convenient rational estimates. The bound
\[
\frac{2d}{(d-1)(d+2-2\sqrt{d-1})}\leq\frac3{d-2}
\]
follows, on writing $y=\sqrt{d-1}$ and clearing positive denominators, from $y^2(y-3)^2+2\geq0$. Therefore
\[
A_d\geq\frac{d^2-5d+3}{(d-2)^2}.
\]

Next put $q=(d-2)/d^2$. Since $q<1/(2\sqrt{d-1})$, the number $\tilde{\beta}(q)$ is the smaller of the two real zeros of
\[
p_q(z):=(d-1)qz^2-z+q.
\]

Since
\[
p_q\left(\frac{d-3}{d(d-2)}\right)
=-\frac{3(d^2-5d+3)}{d^4(d-2)}<0,
\]
we have $\tilde{\beta}(q)<(d-3)/(d(d-2))$, and consequently
\[
R_d\geq r_d:=\frac{d^2-3d+3}{d^2}.
\]

Thus
\[
F_d\geq K_d:=J_d a_d,
\qquad
J_d:=\frac{d(d^2-5d+3)}{8(d-2)(d-1)^2},
\qquad
a_d:=r_d^{d-1}.
\]

For real $x\geq7$, direct differentiation gives
\[
J'(x)=
\frac{x^4+4x^3-19x^2+20x-6}
{8(x-2)^2(x-1)^4}>0.
\]

On the other hand, $a(x):=(1-3/x+3/x^2)^{x-1}$ is decreasing for $x\geq6$ and tends to $e^{-3}$. Indeed, applying $\log(1-u)\leq-u-u^2/2$ with $u=3(x-1)/x^2$ yields
\[
\frac{\mathrm d}{\mathrm dx}\log a(x)
\leq-\frac{3(x-1)(x^3-6x^2+18x-9)}
{2x^4(x^2-3x+3)}<0.
\]

Direct cross-multiplication gives
\[
K_8=\frac9{98}\left(\frac{43}{64}\right)^7>\frac{1}{180}.
\]
For $9\leq d\leq11$, monotonicity gives
\[
K_d\geq J_9a_{11}
=\frac{351}{3584}\left(\frac{91}{121}\right)^{10}>\frac{1}{180},
\]
whereas for $12\leq d\leq14$ it gives
\[
K_d\geq J_{12}a_{14}
=\frac{261}{2420}\left(\frac{157}{196}\right)^{13}>\frac{1}{180}.
\]

For the remaining values $d\geq15$, we use
\[
e<\sum_{k=0}^{5}\frac1{k!}
+\frac1{6!}\sum_{j=0}^{\infty}\frac1{7^j}
=\frac{11743}{4320}<\frac{68}{25}.
\]
The monotonicity and limit above therefore give
\[
K_d\geq J_{15}e^{-3}
>\frac{2295}{20384}\left(\frac{25}{68}\right)^3>\frac{1}{180}.
\]

We have thus shown $F_d>\nicefrac{1}{180}$ for every $d\geq8$.

For the remaining value $d=7$,
\[
A_7=\frac45-\frac7{15(9-2\sqrt6)},
\qquad B_7=\frac{35}{288},
\qquad R_7=\frac{11+\sqrt{1801}}{84}.
\]

Because
\[
\frac{2449}{1000}<\sqrt6<\frac{49}{20},
\qquad
\frac{10609}{250}<\sqrt{1801}<\frac{42439}{1000},
\]
we have
\[
\frac{422}{615}<A_7<\frac{3016}{4395},
\qquad
\frac{4453}{7000}<R_7<\frac{17813}{28000}.
\]
A final cross-multiplication now gives
\[
\frac1{181}
<\frac{422}{615}\frac{35}{288}\left(\frac{4453}{7000}\right)^6
<F_7
<\frac{3016}{4395}\frac{35}{288}\left(\frac{17813}{28000}\right)^6
<\frac{1}{180}.
\]

Thus $d=7$ is the strict minimiser at $\lambda=\tfrac{1}{d-2}$ and moreover this proves the claim.
\end{proof}

\end{document}